\documentclass[a4paper,12pt]{article}
\usepackage{amsmath,amssymb,enumerate,amsthm}
\newcommand{\linktojournal}[1]{\relax}

\usepackage{hyperref}
\hypersetup{
setpagesize=false,
 bookmarksnumbered=true,%
 bookmarksopen=true,%
 colorlinks=true,%
 linkcolor=blue,
 citecolor=red
}

\newcommand{\R}{\mathbb{R}}

\newcommand{\N}{\mathbb{N}}
\newcommand{\ep}{\varepsilon}
\newcommand{\pa}{\partial}

\DeclareMathOperator{\supp}{supp}

\DeclareMathOperator*{\esssup}{esssup}
\newcommand{\lr}[1]{{}\langle{}#1{}\rangle{}}

\newcommand{\rev}[1]{{\color{red}{}#1{}}}

\newtheorem{theorem}{Theorem}[section]
\newtheorem{lemma}[theorem]{Lemma}
\newtheorem{proposition}[theorem]{Proposition}

\theoremstyle{remark}
\newtheorem{remark}{Remark}[section]
\theoremstyle{definition}

\newtheorem{definition}{Definition}[section]

\numberwithin{equation}{section}

\makeatletter
\def\@cite#1#2{[{{\bfseries #1}\if@tempswa , #2\fi}]}
\makeatother                  

\begin{document}
\begin{center}
\Large{{\bf
Optimal decay for one-dimensional damped \\
wave equations with potentials via \\
a variant of Nash inequality
}}
\end{center}

\vspace{5pt}
\begin{center}
Motohiro Sobajima%
\footnote{
Department of Mathematics, 
Faculty of Science and Technology, Tokyo University of Science,  
2641 Yamazaki, Noda-shi, Chiba, 278-8510, Japan,  
E-mail:\ {\tt msobajima1984@gmail.com}}
\end{center}

\newenvironment{summary}{\vspace{.5\baselineskip}\begin{list}{}{%
     \setlength{\baselineskip}{0.85\baselineskip}
     \setlength{\topsep}{0pt}
     \setlength{\leftmargin}{12mm}
     \setlength{\rightmargin}{12mm}
     \setlength{\listparindent}{0mm}
     \setlength{\itemindent}{\listparindent}
     \setlength{\parsep}{0pt}
     \item\relax}}{\end{list}\vspace{.5\baselineskip}}
\begin{summary}
{\footnotesize {\bf Abstract.}
The optimality of decay properties of the one-dimensional damped wave equations with potentials belonging to a certain class is discussed. 
The typical ingredient is a variant of Nash inequality which involves an invariant measure for the corresponding Schr\"odinger semigroup. 
This enables us to find a sharp decay estimate from above. 
Moreover, the use of a test function method with the Nash-type inequality provides the decay estimate from below.  
The diffusion phenomena for the damped wave equations with potentials are also considered. 
}
\end{summary}



\tableofcontents
\newpage
\section{Introduction}

As is well-known, solutions of the usual damped wave equation behave like those of the heat equation. 
This phenomenon is so-called diffusion phenomenon. 
If we consider damped wave equations with potentials, then the situation seems similar, 
that is, solutions of damped wave equations with potentials can be expected to have an asymptotic behavior similar to those of parabolic equations governed by Schr\"odinger operators. 
However, for the theory of Schr\"odinger operators, 
the large-time behavior of the corresponding semigroup 
can be affected by adding a certain potential even if the size of potential is sufficiently small.
This means that the Schr\"odinger semigroup has a different behavior 
compared with the heat semigroup. 
Our interest is to find this kind of significant change 
in the analysis of damped wave equations. 

To this aim, in this paper, we consider the Cauchy problem of 
one-dimensional linear damped wave equations of the form
\begin{equation}\label{problem}
\begin{cases}
\pa_t^2u(x,t)-\pa_x^2u(x,t)+\pa_tu(x,t)+V(x)u(x,t)=0 
&\text{in}\ \R\times (0,\infty), 
\\
(u,\pa_tu)(\cdot,0)=(u_0,u_1)\in \mathcal{H}=H^1(\R)\times L^2(\R),
\end{cases}
\end{equation}
where $\pa_t=\frac{\pa}{\pa t}$, $\pa_x=\frac{\pa}{\pa x}$, 
and $L^2(\R)$ and $H^1(\R)$ is the usual Lebesgue and Sobolev spaces, respectively. 
The coefficient $V$ is assumed to belong 
the following class:
\begin{equation}\label{intro:V:ass}
\mathcal{V}=\left\{V\in BC(\R)\;;\;V\geq 0\ \&\ 0<\int_{\R} |x|V(x)\,dx <+\infty\right\},
\end{equation}
where $BC(\R)$ stands for the set of all bounded continuous functions in $\R$. 
The simplest examples of $\mathcal{V}$ are 
non-trivial nonnegative functions belonging 
to $C_0^\infty(\R)$ (compactly supported smooth functions). 
The functions $\lr{x}^{-\alpha}$ with $\alpha>2$ also belong to 
the class $\mathcal{V}$, where $\lr{x}=(1+x^2)^{1/2}$ $(x\in \R)$. 
The goal of this paper is 
to clarify the large time behavior of the solution $u$ to \eqref{problem}
and also that of 
the total (local) energy functional defined as 
\begin{align}
\label{intro:energy:V}
E(u;t)&=\int_\R\mathcal{E}(u;x,t)\,dx 
\quad 
\left(E_R(u;t)=\int_{-R}^R\mathcal{E}(u;x,t)\,dx\right),
\end{align}
where $\mathcal{E}(u;\cdot,t)$ is the energy density 
\begin{align}
\label{intro:energy-density:V}
\mathcal{E}(u;x,t)&=(\pa_tu(x,t))^2+(\pa_xu(x,t))^2+V(x)u(x,t)^2.
\end{align}

The $N$-dimensional usual damped wave equation $(V\equiv 0)$ 
\begin{equation}\label{Dd-damped}
\begin{cases}
\pa_t^2u(x,t)-\Delta u(x,t)+\pa_tu(x,t)=0 
&\text{in}\ \R^N\times (0,\infty), 
\\
(u,\pa_tu)(\cdot,0)=(u_0,u_1) 
\end{cases}
\end{equation}
has a large mount of literature. 
The paper \cite{Matsumura1976} by Matsumura 
is the pioneering work in this field. He proved a series of 
inequalities (Matsumura estimates) describing the diffusive structure 
of the damped wave equation 
which is close to that of heat equation 
\begin{equation}\label{Dd-heat}
\begin{cases}
\pa_t v(x,t)-\Delta v(x,t)=0 
&\text{in}\ \R^N\times (0,\infty), 
\\
v(0)=v_0.
\end{cases}
\end{equation}
For instance, by using these inequalities 
one can find the energy decay structure of \eqref{Dd-damped}
as follows:
\[
\|\pa_tu\|_{L^2(\R^N)}^2
+
\|\nabla u\|_{L^2(\R^N)}^2\leq C(1+t)^{-N/2-1}
\Big(\|u_0\|_{H^1\cap L^1(\R^N)}^2+\|u_1\|_{L^2\cap L^1(\R^N)}^2\Big).
\]
Actually it is well-known that 
diffusion phenomena for \eqref{Dd-damped} occur, that is, 
the solution of \eqref{Dd-damped} is asymptotically equivalent to 
the one of \eqref{eq:heat} with $v_0=u_1+u_1$ (see e.g., Hsiao--Liu \cite{HsLi1992} and Yan--Milani \cite{YM2000}). 
For the multi-dimensional case, 
the exterior problems with a suitable boundary condition 
has been dealt with (see Dan--Shibata \cite{DaSh1995}, 
Ikehata--Matsuyama \cite{IkMa2002}, Ikehata \cite{Ikehata2002} 
and also Sobajima--Wakasugi \cite{SoWa2016} and references therein). 
Also in these case, the diffusion phenomenon occurs. 
Of course the asymptotic profile is affected by the boundary. 
On the one hand, in the parabolic sense, the boundary can be understood 
as an absorbing wall. On the other hand, 
in the hyperbolic sense, the one behaves as the reflection wall which makes some difficulty in the analysis of the waves. 
With a different aspect, 
the Cauchy problem of wave equations  
with space- or time-dependent coefficients (damping and potential) 
has already been discussed so far.

Our interest in the present paper is the diffusive structure for the case where damped wave equations equip potentials.
More precisely, we mainly discuss the case where any nonnegative (non-zero) potentials 
cannot be regarded as small perturbations for the Laplacian in the sense of the large-time behavior. 
In the field of stochastic analysis on manifolds, 
such a situation can be explained via the recurrence 
of the Brownian motion in manifolds
(see e.g., Grigor'yan and Saloff-Coste \cite{GS2002}).
If we are restricted to the Euclidean spaces $\R^N$, 
the cases $N=1$ and $N=2$ are suitable.
Here we focus our attention to the one-dimensional case. 
In Matsumura \cite{Matsumura1976}, it has been found (as explained above) that 
solutions of the one-dimensional damped wave equation satisfy
the energy decay
\[
\|\pa_tu\|_{L^2(\R)}^2
+
\|\nabla u\|_{L^2(\R)}^2= O(t^{-3/2})
\]
as $t\to \infty$, which reaches the optimal decay rate for the one-dimensional heat equation. 

Recently, 
Ikehata  \cite{Ikehata_arXiv} considered the problem 
\eqref{problem} with positive potentials.
Roughly speaking, he has succeeded in proving even if $V$ is small enough (but positive),  
the corresponding energy of solutions with compactly supported initial data 
decays like 
\[
E(u;t)\leq Ct^{-2}\Big(\|u_0\|_{H^1(\R)}^2+\|u_1\|_{L^2(\R)}^2\Big)
\]
which is faster than that of solutions of the usual damped wave equation obtained by Matsumura. 
This kind of phenomenon does not appear so far. 
This significant change may be understood as the influence of potentials. 
In Ikehata--Li \cite{IkLi_arXiv}, 
a similar analysis for the following damped wave equations with space-dependent (effective) damping
is also considered:
\begin{equation}\label{problem-IkLi}
\begin{cases}
\pa_t^2u(x,t)-\pa_x^2u(x,t)+a(x)\pa_tu(x,t)+V(x)u(x,t)=0 
&\text{in}\ \R\times (0,\infty), 
\\
(u,\pa_tu)(\cdot,0)=(u_0,u_1).
\end{cases}
\end{equation}
Inspired by the results \cite{Ikehata_arXiv} (and \cite{IkLi_arXiv}), 
we would investigate 
the constant damping case with general potentials belonging to the class $\mathcal{V}$.  

Before stating our result, we clarify 
the definition of solutions to the problem \eqref{problem}. 

\begin{definition}\label{def:ops-sols}
\begin{itemize}
\item[\bf (i)] We define the Schr\"odinger operator $S=-\frac{d^2}{dx^2}+V$ 
acting in $L^2(\R)$ as 
\[
S=S_0+V \quad D(S)=D(S_0) 
\quad \text{with}\ S_0=-\frac{d^2}{dx^2}, \quad D(S_0)=H^2(\R);
\]
note that if $V\in\mathcal{V}$, then 
$S$ is nonnegative and selfadjoint in $L^2(\R)$ and 
the domain of its square root can be characterized as $D(S^{1/2})=H^1(\R)$.
\item[\bf (ii)] We call $u:\R\times [0,\infty)\to \R$ is a weak solution  of \eqref{problem} if 
$u\in C^1([0,\infty);L^2(\R))\cap C([0,\infty);H^1(\R))$ satisfies
\[
(u(t),\pa_tu(t))=e^{t\mathcal{A}}(u_0,u_1), \quad t\geq 0, 
\]
where $e^{t\mathcal{A}}$ is the $C_0$-semigroup on $\mathcal{H}=H^1(\R)\times L^2(\R)$ 
generated by $\mathcal{A}(u,v)=(v,-Su-v)$ endowed with domain $D(\mathcal{A})=H^2(\R)\times H^1(\R)$. 
\end{itemize}
\end{definition}

This immediately gives 
the wellposedness of the problem \eqref{problem} 
via the Hille--Yosida theorem to the abstract Cauchy problem 
in a real Hilbert space $H$ of the form
\begin{align}
\label{intro:abst}
\begin{cases}
z''(t)+Az(t)+z'(t)=0, \quad t>0,
\\
(z,z')(0)=(z_0,z_1)\in D(A^{1/2})\times H
\end{cases}
\end{align}
with a nonnegative selfadjoint operator $A$. 
Since such an abstract problem has the corresponding energy functional
$\|z'(t)\|_H^2+\|A^{1/2}z(t)\|_H^2$, 
we would deal with the energy functional of 
the form \eqref{intro:energy:V}
with \eqref{intro:energy-density:V}.

Now we are in a position to state our result of this paper. 
The first describes a diffusion phenomenon 
and a weighted energy decay estimates. 
Some significant change appears compared with the case $V\equiv 0$.
\begin{theorem}\label{main1}
Assume that $V\in\mathcal{V}$ and 
the pair $(u_0,u_1)\in \mathcal{H}$ satisfies $\lr{x}(u_0+u_1)\in L^1(\R)$. 
Let $u$ be the weak solution of \eqref{problem}. 
Then there exists a positive constant $C$ such that 
for every $t\geq 1$, 
\begin{align*}
\|u(t)-e^{-tS}(u_0+u_1)\|_{L^2(\R)}
&\leq Ct^{-1} \Big(\|u_0\|_{H^1(\R)}+\|u_1\|_{L^2(\R)}\Big), 
\\
\|e^{-tS}(u_0+u_1)\|_{L^2(\R)}
&\leq Ct^{-3/4} \|\lr{x}(u_0+u_1)\|_{L^1(\R)}.
\end{align*}
Moreover, for every $\beta\in[0,1)$, there exists a positive constant $C_\beta$ such that 
for every $t\geq 0$, 
\[
\int_{\R}\frac{\mathcal{E}(u;x,t)}{\lr{x}^{\beta}}\,dx
\leq 
C_\beta
(1+t)^{-(5+\beta)/2}
\Big(\|u_0\|_{H^1(\R)}^2+\|u_1\|_{L^2(\R)}^2+\|\lr{x}(u_0+u_1)\|_{L^1(\R)}^2\Big).
\]
\end{theorem}
\begin{remark}
The couple of the first estimates describes 
the appearance of diffusion phenomena. 
We emphasize that the asymptotic profile $e^{-tS}(u_0+u_1)$ 
has a different behavior compared with the usual damped wave equation 
$\|e^{-tS_0}(u_0+u_1)\|_{L^2}=O(t^{-1/4})$ as $t\to \infty$.
\end{remark}
\begin{remark}
For the second estimate, 
the choice $\beta=0$ gives the decay estimate for the total energy 
\[
E(u;t)=O(t^{-5/2})
\]
as $t\to \infty$. Therefore we could find a total energy decay estimate which is 
faster than those of the results in Ikehata \cite{Ikehata_arXiv} and Ikehata--Li \cite{IkLi_arXiv}.
On the other hand, 
by the different choice $\beta=1-2\ep$ with $\ep\ll 1$ 
we can find 
an estimate describing the local energy decay 
faster than total energy decay:
for every $R>0$, 
\[
E_R(u;t)=O(t^{-3+\ep})
\]
as $t\to \infty$. 
From the viewpoint of diffusion phenomena, 
it is remarkable that 
the Gaussian function $G(x,t)=(4\pi t)^{-1/2}e^{-\frac{x^2}{4t}}$, 
as the typical solution of the heat equation, 
satisfies for every $R>0$, 
\begin{align*}
t^3\|\pa_x G(\cdot,t)\|_{L^2((-R,R))}^2
&=
\frac{1}{16\pi}\int_{-R}^Rx^2e^{-\frac{x^2}{2t}}\,dx
\to 
\frac{R^3}{24\pi},
\\
t^5\|\pa_t G(\cdot,t)\|_{L^2((-R,R))}^2
&=
\frac{1}{64\pi }\int_{-R}^R x^4e^{-\frac{x^2}{2t}}\,dx
\to 
\frac{R^5}{160\pi}
\end{align*}
as $t\to \infty$. Although we do not know whether such a local energy decay 
$E_R(u;t)=O(t^{-3+\ep})$
is sharp or not, 
the strategy of the proof of Theorem \ref{main1} might be 
a suggestion of viewpoints 
for the analysis of local energy decay of hyperbolic equations.
\end{remark}

Next we consider the optimality of the total energy decay 
obtained in Theorem \ref{main1}.

\begin{theorem}\label{main2}
Assume that $V\in\mathcal{V}$ 
and 
the pair $(u_0,u_1)\in \mathcal{H}$ satisfies $\lr{x}(u_0+u_1)\in L^1$. 
If $u_0+u_1$ is non-trivial and nonnegative, then 
there exist positive constants $c$ and  $c'$ such that 
for every $t\geq 0$, 
\begin{gather*}
\int_{0}^t(1+\tau)^{1/2}\|u(\tau)\|_{L^2}^2\,d\tau\geq c\log (1+t), 
\\
\int_{0}^t(1+\tau)^{3/2}E(u;\tau)\,d\tau\geq c'\log (1+t).
\end{gather*}
\end{theorem}
\begin{remark}
Theorem \ref{main1} yields the opposite side 
of the estimate in Theorem \ref{main2}:
\[
\int_{0}^t(1+\tau)^{3/2}E(u;\tau)\,d\tau\leq 
C_0
\Big(\|u_0\|_{H^1}^2+\|u_1\|_{L^2}^2+\|\lr{x}(u_0+u_1)\|_{L^1}^2\Big)\log (1+t).
\]
Therefore Theorem \ref{main2} assures that the decay rate $5/2$ of 
the total energy is optimal. 
\end{remark}

Here we would briefly explain the concept and idea 
in this paper. 
Our approach depends on the fundamental result 
for the abstract evolution equation \eqref{intro:abst} 
provided by Radu--Todorova--Yordanov \cite{RaToYo2011}: 
\begin{align*}
\big\|A^{1/2} (\Theta(t) g-e^{-tA}g)\big\|_{H}
\leq 
C\Big(
t^{-3/2}\|e^{-\frac{t}{2}A}g\|_{H}
+e^{-t/16}\|(A^{1/2}+1)^{-1}A^{1/2}g\|_H
\Big),
\end{align*}
where $\Theta (\cdot)g$ is the solution of \eqref{intro:abst} with $(z_0,z_1)=(0,g)$.
Then it becomes clear that 
the estimate for $e^{-tA}$ plays a crucial role.
In our situation, 
the profile of the Schr\"odinger semigroup $e^{-tS}$ 
is essential. 
In the criticality theory of Schr\"odinger operators (developed by Murata \cite{Murata1984}), 
the operator $S_0=-\frac{d^2}{dx^2}$ is critical 
which means that $S_0$ is unstable under the small perturbation with nonnegative potentials. 
This can be explained by the (non-)uniqueness of positive harmonic function (modulo constants)
with respect to $\mathcal{S}=-\frac{d^2}{dx^2}+V$ in the sense of distribution $\mathcal{D}'(\R)$;
\begin{equation}\label{intro:harmonics}
\mathcal{S}\psi =-\psi''+V\psi=0\quad \text{in}\ \mathcal{D}'(\R).
\end{equation}
If $V\in \mathcal{V}$, then \eqref{intro:harmonics} has two linearly independent positive solutions 
$\psi_1$ and $\psi_2$. On the one hand, 
by using $\psi_1$ and $\psi_2$, by a potential theory-like technique
we can prove a couple of Hardy-Rellich type inequalities 
for the Schr\"odinger operator $S$. 
On the other hand, the positive harmonic function $\psi$ with respect to $\mathcal{S}$ 
also provides so-called invariant measure $d\mu=\psi(x)dx$ for the semigroup $e^{-tS}$:
\[
\int_{\R}e^{-tS}f\,d\mu
=\int_{\R}f\,d\mu, \quad f\in L^2(\R)\cap L^1(\R,d\mu),\ t>0.
\]
Since the sharp $L^2$-decay estimate for the $N$-dimensional heat semigroup $e^{t\Delta}$ 
is provided via the Nash inequality
\[
\|f\|_{L^2(\R^N)}^{2+4/N}\leq C_{\rm Nash}\|f\|_{L^1(\R^N)}^{4/N}\|\nabla f\|_{L^2(\R^N)}^2, \quad f\in H^1(\R^N)\cap L^1(\R^N)
\] 
(for the best possible constant $C_{\rm Nash}$ see e.g., Carlen--Loss \cite{CaLo1993} and also Lemma \ref{lem:nash})
and $dx$ can be regarded as the invariant measure for $e^{t\Delta}$, 
it is natural to move to discuss the validity of 
the Nash-type inequality of the form 
\[
\|f\|_{L^2(\R)}^{2+\gamma}
\leq 
C
\|f\|_{L^1(\R,d\mu)}^{\gamma}\|S^{1/2}f\|_{L^2(\R)}^2
\quad 
f\in H^1(\R)\cap L^1(\R,d\mu).
\]
This inequality enables us to find a $L^2$-decay estimate for the Schr\"odinger semigroup $e^{-tS}$ as 
\[
\|e^{-tS}f\|_{L^2(\R)}\leq Ct^{-1/\gamma}\|f\|_{L^1(\R,d\mu)}, \quad t>0.
\]
In our situation, the parameter $\gamma$ in the above estimate 
can be chosen as $\gamma=4/3$ 
which is surprisingly different from $\gamma=4$ (the case with $V\equiv 0$). 
This provides the decay estimate of the Dirichlet energy for 
the semigroup $e^{-tS}$ as 
\[
\|S^{1/2}e^{-tS}f\|_{L^2(\R)}^2\leq Ct^{-5/2}\|f\|_{L^1(\R,d\mu)}^2, \quad t>0.
\]
The above estimate finally employed to the energy decay of 
the damped wave equation with the potential $V\in\mathcal{V}$. 
Several estimates with weights vanishing at the spatial infinity $(\beta>0)$ 
are consequences of the validity of Hardy--Rellich inequalities with a parameter.
To prove the energy estimates from below, we employ 
so-called {\it test function method}
which is well-known as a technique for proving the blowup phenomena 
for some nonlinear problem with a special structure. 
Our argument is the use of the test function method with positive harmonic functions with respect to $\mathcal{S}$ for the {\it linear} problem 
with a variant of the Nash inequality. 
Roughly speaking, in the case without potential, 
if $L^1$-norm of $u(t)$ is completely known, then the Nash inequality enables us 
to find the lower bound of $\|\nabla u(t)\|_{L^2}$ via the 
lower bound of the $L^2$-norm of $u(t)$. 
The lower bound of some integral of $|u(t)|$ can be derived via the test function method. 
The proof of Theorem \ref{main2} is done by a suitable modification of the strategy explained above. 

Throughout in this paper after introduction, 
we focus our attention to the case of one-dimension 
and use the following notations for shorten the notation:
\[
\|f\|_{L^p}=
\|f\|_{L^p(\R)}=
\begin{cases}
\displaystyle 
\left(\int_{\R}|f|^p\,dx\right)^{1/p} & \text{if}\ 1\leq p<\infty, 
\\[3pt]
\displaystyle 
\esssup_{x\in\R}|f(x)|, & \text{if}\ p=\infty.
\end{cases}
\]
The present paper is organized as follows: 
In Section \ref{sec:pre}, we collect the fundamental facts, 
especially, several facts in ordinary differential equations, 
the Nash inequality in one-dimension 
and the diffusion phenomenon for the abstract evolution equation related to the damped wave equation.
In Section \ref{sec:hardy-rellich-nash}, 
we show the validity of Hardy-Rellich inequalities 
with respect to the Schr\"odinfer operator $S$. 
A variant of Nash inequality is also introduced in Section \ref{sec:hardy-rellich-nash}. 
In Section \ref{sec:semigroup},
the upper and lower bounds of the Dirichlet energy $\|S^{1/2}e^{-tS}f\|_{L^2}^2$ 
are provided. 
Then finally in Section \ref{sec:damped} we discuss the energy decay 
of solutions to \eqref{problem} and its optimality.

\section{Preliminaries}\label{sec:pre}

First we recall the fundamental fact of 
the theory of ordinary differential equations. 
We focus our attention to 
the homogeneous problem 
\begin{equation}\label{eq:harmonic:V}
\mathcal{S}\psi=-\psi''+V\psi=0, \quad x\in \R.
\end{equation}
The following lemma gives 
a well-behaved fundamental system of \eqref{eq:harmonic:V}
(see e.g., Olver \cite{Olver} and also Metafune--Sobajima \cite{equadiff2017}).
\begin{lemma}\label{lem:ode:V}
Assume that $V\in \mathcal{V}$. 
Then there exist 
an increasing solution $\psi_1$
and 
a decreasing solution $\psi_2$ of \eqref{eq:harmonic:V}
satisfying 
\begin{gather}\label{eq:asym-psi}
\lim_{x\to-\infty}\big(\psi_1(x),x\psi_1'(x)\big)=(1,0), 
\quad
\lim_{x\to+\infty}\big(\psi_2(x),x\psi_2'(x)\big)=(1,0)
\end{gather}
with the Wronskian 
$k_V\equiv \psi_1'\psi_2-\psi_1\psi_2'>0$. 
Moreover, one has 
\begin{gather*}
\lim_{x\to+\infty}\left(\frac{\psi_1(x)}{x},\psi_1'(x)\right)=(k_V,k_V), 
\quad
\lim_{x\to-\infty}\left(\frac{\psi_2(x)}{x},\psi_2'(x)\right)=(-k_V,-k_V)
\end{gather*}
and 
$\|\psi_1'\|_{L^\infty}=
\|\psi_2'\|_{L^\infty}=
\|\psi_1'\psi_2+\psi_1\psi_2'\|_{L^\infty}=k_V$.
\end{lemma}
\begin{remark}
Existence of the prescribed asymptotic behavior \eqref{eq:asym-psi} is
valid under the integrability of $xV(x)$. 
If $V\geq 0$ and the pair $(\psi_1,\psi_2)$ is not linearly independent, then 
the convexity of $\psi_1$ gives $\psi_1=\psi_2\equiv 1$. This case is nothing but $V\equiv 0$
which is removed from the class $\mathcal{V}$.
\end{remark}

Next we recall the Nash inequality in one-dimension. 
The best possible constant is found in Carlen--Loss {\cite{CaLo1993} 
for all dimensions.
\begin{lemma}[Carlen--Loss {\cite{CaLo1993}}]\label{lem:nash}
For every $f\in H^1(\R^N)\cap L^1(\R^N)$,  
\[
\|f\|_{L^2}^{2+4/N}\leq C_{\rm Nash}\|f\|_{L^1}^{4/N}\|f'\|_{L^2}^2,
\]
where $C_{\rm Nash}$ is the best possible constant given by 
\[
C_{\rm Nash}
=\frac{N+2}{N \lambda_N}\left(\frac{N+2}{2\omega_N}\right)^{2/N},
\]
where $\omega_N$ is volume of unit ball $B_N(0,1)$ in $\R^N$ 
and $\lambda_{N}$ is the smallest positive eigenvalue 
of the negative Laplacian with the Neumann boundary condition in $L_{\rm rad}^2(B_N(0,1))$.
\end{lemma}

\begin{remark}
If we do not care about the best possible constant, 
the one-dimensional case has an elementally proof. If $f\in C_0^1(\R)$, then 
it is easy to check
\begin{align*}
|f(x)|^2
\leq 2\min\left\{\int_{-\infty}^xf'f\,dy,\int_{x}^\infty f'f\,dy\right\}
\leq \int_{-\infty}^\infty|f'||f|\,dy
\leq \|f'\|_{L^2}\|f\|_{L^2}. 
\end{align*}
Combining the above inequality with the H\"older inequality, we have
\[
\|f\|_{L^2}^8
\leq \|f\|_{L^1}^{4}\|f\|_{L^\infty}^{4}
\leq \|f\|_{L^1}^{4}\|f'\|_{L^2}^{2}\|f\|_{L^2}^{2}.
\]
\end{remark}
We also employ the abstract Matsumura estimates 
(essentially) proved in Radu--Todorova--Yordanov \cite{RaToYo2011}. 
To apply their result to our situation, 
we also introduce 
the abstract Cauchy problem of the second-order 
differential equation 
in a Hilbert space $H$ 
governed by a nonnegative selfadjoint operator $A$:
\begin{align}
\label{eq:abst}
\begin{cases}
z''(t)+Az(t)+z'(t)=0, \quad t>0,
\\
(z,z')(0)=(0,g), \quad g\in H.
\end{cases}
\end{align}
The following lemma is the abstract version of 
the statement for diffusion phenomena. 

\begin{lemma}[{Radu--Todorova--Yordanov \cite{RaToYo2011}}]\label{lem:matsumura-abst}
Let $\Theta(\cdot)g$ be the unique solution of \eqref{eq:abst}. 
Then the following assertions hold:
\begin{itemize}
\item[\bf (i)]
There exists a positive constant $C_{{\rm M},1}$ such that 
for every $g\in H$ and $t\geq 1$, 
\begin{align*}
\big\|\Theta(t)g-e^{-tA}g\big\|_{H}
\leq 
C_{{\rm M},1}\Big(
t^{-1}\|e^{-\frac{t}{2}A}g\|_{H}
+e^{-t/16}\|(A^{1/2}+1)^{-1}g\|_H
\Big).
\end{align*}
\item[\bf (ii)]
There exists a positive constant $C_{{\rm M},2}$ such that 
for every $g\in H$ and $t\geq 1$, 
\begin{align*}
\big\|A^{1/2} (\Theta(t)g-e^{-tA}g)\big\|_{H}
\leq 
C_{{\rm M},2}\Big(
t^{-3/2}\|e^{-\frac{t}{2}A}g\|_{H}
+e^{-t/16}\|(A^{1/2}+1)^{-1}A^{1/2}g\|_H
\Big).
\end{align*}
\item[\bf (iii)]
There exists a positive constant $C_{{\rm M},3}$ such that 
for every $g\in H$ and $t\geq 1$, 
\begin{align*}
\left\|\frac{d}{dt}\big(\Theta(t)g-e^{-tA}g\big)\right\|_{H}
\leq 
C_{{\rm M},3}t^{-2}\Big(
\|g\|_{H}
+e^{-t/4}\|g\|_H
\Big).
\end{align*}
\end{itemize}
\end{lemma}
\section{A variant of Hardy-Rellich and Nash inequality}
\label{sec:hardy-rellich-nash}
\subsection{A variant of Hardy-Rellich inequality for $S$}
The following is a kind of Hardy inequality in one-dimension 
which is valid only if the Schr\"odinger operator 
has a non-negative (non-zero) potential.
\begin{lemma}\label{lem:hardy:V}
Assume that $V\in \mathcal{V}$. 
Let the pair $(\psi_1,\psi_2)$ be given in Lemma \ref{lem:ode:V}. 
Then the following assertions hold:
\begin{itemize}
\item[\bf (i)]
For every $f\in D(S^{1/2})=H^1(\R)$,
\[
\left\|\frac{f}{\psi_1\psi_2}\right\|_{L^2}\leq 
C_{1,1}
\|S^{1/2}f\|_{L^2}, 
\quad 
C_{1,1}=\frac{2}{k_V}.
\]
\item[\bf (ii)]
For every $\beta\in [0,1)$, 
and $f\in D(S)=H^2(\R)$, 
\begin{align*}
\left\|\frac{f}{(\psi_1\psi_2)^{1+\beta/2}}\right\|_{L^2}
\leq 
C_{1,2}
\|Sf\|_{L^2}^{(2+\beta)/4}
\|f\|_{L^2}^{(2-\beta)/4}, 
\quad 
C_{1,2}=\left(\frac{4}{(1-\beta^2)k_V^2}\right)^{(2+\beta)/4}.
\end{align*}
\item[\bf (iii)]
For every $\beta\in [0,1)$ and $f\in D(S)$, 
\[
\int_{\R}\frac{|f'|^2+Vf^2}{(\psi_1\psi_2)^{\beta}}\,dx
\leq 
C_{1,3}\|Sf\|_{L^2}^{1+\beta/2}
\|f\|_{L^2}^{1-\beta/2}, 
\quad 
C_{1,3}=\frac{2-\beta}{2(1-\beta)}
\left(\frac{4}{(1-\beta^2)k_V^2}\right)^{\beta/2}.
\]
\end{itemize}
\end{lemma}
\begin{proof}
Set
\[
\psi_*(x)=\psi_1(x)^{1/2}\psi_2(x)^{1/2}, \quad x\in\R.
\]
Then 
by a direct calculation we see that $\psi_*$ is a positive super-harmonic function with respect to 
$\mathcal{S}$ with 
\begin{equation}\label{eq:psi_*}
\mathcal{S}\psi_*
=\frac{(\psi_1'\psi_2-\psi_1\psi_2')^2}{4(\psi_1\psi_2)^{3/2}}
=\frac{k_V^2}{4\psi_*^3}>0.
\end{equation}
{\bf (i)}\ 
By density, it suffices to show the validity of the inequality 
for $f\in C_0^1(\R)$. 
Putting $g=\psi_*^{-1}f\in C_0^1(\R)$,
we see from integration by parts that 
\begin{align*}
\int_\R |f'|^2+Vf^2\,dx
&=
\int_\R (|g'|^2+Vg^2)\psi_*^2\,dx
+
2\int_\R gg'\psi_*\psi_*'\,dx
+
\int_\R |\psi_*'|^2g^2\,dx
\\
&=
\int_\R |g'|^2\psi_*^2\,dx
+\int_\R g^2\psi_*(\mathcal{S}\psi_*)\,dx.
\end{align*}
By using \eqref{eq:psi_*}, 
we deduce the desired inequality.

{\bf (ii)}
It is enough to 
consider the case 
$f\in C_0^2(\R)$. 
Putting $g_\beta=\psi_*^{-\beta}f$, 
we see form 
integration by parts that 
\begin{align*}
\int_{\R}\frac{(Sf)f}{\psi_*^{2\beta}}\,dx
&=
-\int_{\R}(\psi_*^\beta g_\beta)''\psi_*^{-\beta}g_\beta\,dx
+
\int_{\R}Vg_\beta^2\,dx
\\
&=
\int_{\R}(\psi_*^\beta g_\beta)'(\psi_*^{-\beta}g_\beta)'\,dx
+
\int_{\R}Vg_\beta^2\,dx
\\
&=
\int_{\R}|g_\beta'|^2+Vg_\beta^2\,dx
-\beta^{2}\int_{\R}g_\beta^2\frac{|\psi_*'|^2}{\psi_*^2}\,dx.
\end{align*}
Using {\bf (i)} and $\|\psi_*\psi_*'\|_{L^\infty}=\frac{1}{2}\|\psi_1'\psi_2+\psi_1\psi_2'\|_{L^\infty}=k_V/2$, we deduce
\[
\int_{\R}\frac{(Sf)f}{\psi_*^{2\beta}}\,dx
\geq 
\frac{k_V^2}{4}
\int_{\R}\frac{g_\beta^2}{\psi_*^4}\,dx
-\beta^{2}\int_{\R}g_\beta^2\frac{|\psi_*'|^2}{\psi_*^2}\,dx
\geq 
(1-\beta^2)
\frac{k_V^2}{4}
\int_{\R}\frac{g_\beta^2}{\psi_*^4}\,dx.
\]
By the H\"older inequality, we conclude that
\begin{align*}
(1-\beta^2)
\frac{k_V^2}{4}\|\psi_*^{-2-\beta}f\|_{L^2}^2
&\leq 
\|Sf\|_{L^2}\|\psi_{*}^{-2\beta}f\|_{L^2}
\\
&\leq 
\|Sf\|_{L^2}
\|f\|_{L^2}^{\frac{2-\beta}{2+\beta}}
\|\psi_{*}^{-2-\beta}f\|_{L^2}^{\frac{2\beta}{2+\beta}}.
\end{align*}
which provides the desired inequality. 

{\bf (iii)}
By integration by parts and 
$\|\psi_*\psi_*'\|_{L^\infty}=k_V/2$, 
\rev{t}
he second inequality can be deduce as follows:
\begin{align*}
\int_{\R}\frac{|f'|^2+(1+\beta)Vf^2}{\psi_*^{2\beta}}\,dx
&=
\int_{\R}\frac{-f''f+(1+\beta)Vf^2}{\psi_*^{2\beta}}\,dx
+2\beta
\int_{\R}\frac{f'f\psi_*'}{\psi_*^{2\beta+1}}\,dx
\\
&=
\int_{\R}\frac{(Sf)f}{\psi_*^{2\beta}}\,dx
+\beta
\int_{\R}\frac{f^2\mathcal{S}\psi_*}{\psi_*^{2\beta+1}}\,dx
+\beta(2\beta+1)
\int_{\R}\frac{f^2|\psi_*'|^2}{\psi_*^{2\beta+2}}\,dx
\\
&\leq 
\|S_Vf\|_{L^2}
\|f\|_{L^2}^{\frac{2-\beta}{2+\beta}}
\|\psi_{*}^{-2-\beta}f\|_{L^2}^{\frac{2\beta}{2+\beta}}
+\frac{\beta(\beta+1)k_V^2}{2}
\|\psi_{*}^{-2-\beta}f\|_{L^2}^2.
\end{align*}
Combining the above inequality with {\bf (ii)}, we obtain the desired inequality.
The proof is complete.
\end{proof}

\begin{remark}
The inequality in Lemma \ref{lem:hardy:V} {\bf (ii)} with $\beta=1$ is false. 
In this situation we can check that
$\psi_*^2=\psi_1\psi_2$ and $\psi=\psi_1+\psi_2$ (and also $\lr{x}$)
are comparable:
\[
0<
\inf_{x\in\R}\left(\frac{\psi(x)}{\psi_*(x)^2}\right)
\leq
\sup_{x\in\R}\left(\frac{\psi(x)}{\psi_*(x)^2}\right)
<+\infty.
\]
Taking $\zeta\in C_0^\infty(\R)$ with 
${\rm supp}\,\zeta=[1,2]$ and
\[
f_n(x)=\psi(x)\zeta\left(\frac{\psi(x)^{1/2}}{n}\right)\quad (n\in \N),
\]
we can check that 
$\|\psi_*^{-3}f_n\|_{L^2}\approx (\log n)^{1/2}$, 
$\|f_n\|_{L^2}\approx n^3$ and 
$\|Sf_n\|_{L^2}\approx n^{-1}$ for sufficiently large $n$. Therefore the inequality 
with $\beta=1$ is never valid. 
\end{remark}
\subsection{A variant of Nash inequality for $S$}

Now we define a positive harmonic function 
$\psi_V$ with respect to $\mathcal{S}$ 
and the corresponding weighted measure $d\mu_V$ as follows:
\begin{equation}\label{def:mu_V}
\psi_V(x)=\psi_1(x)+\psi_2(x), \quad d\mu_V=\psi_V(x)\,dx.
\end{equation}
Then $d\mu_V$ forms an invariant measure for the 
Schr\"odinger semigroup $e^{-tS}$. 
It is formally seen in the following computation 
\[
\frac{d}{dt}
\int_{\R}e^{-tS}f\,d\mu_V
=
-\int_{\R}S(e^{-tS}f)\psi_V\,dx
=
-\int_{\R}(e^{-tS}f)\mathcal{S}\psi_V\,dx=0.
\]
Although it is valid for every $f\in L^2(\R)\cap L^1(\R,d\mu_V)$, 
to reach our goal, we only use the following fact 
describing a weighted $L^1$-contraction property. 
The proof of the invariant measure  is omitted.
\begin{lemma}\label{lem:C_M-inv}
If $f\in L^2(\R)\cap L^1(\R,d\mu_V)$, 
then one has 
for every $t\geq0$, 
\begin{gather}
\|\psi_V(e^{-tS}f)\|_{L^1}\leq\|\psi_V f\|_{L^1}.
\end{gather}
\end{lemma}
\begin{proof}
Set for $m>0$, 
\[
\mathcal{C}_{m}=\left\{f\in L^2(\R)\;;\;\|\psi_V f\|_{L^1}\leq m \right\}.
\]
Then it suffices to prove that 
for every $m>0$, $\mathcal{C}_m$ is invariant under the flow $e^{-tS}$.

Observe that $\mathcal{C}_m$ is closed in $L^2(\R)$ via Fatou's Lemma.
We set another family of subsets in $L^2(\R)$ as follows:
\begin{align*}
\mathcal{C}_{M}^*
&:=
\left\{f\in L^2(\R)\;;\; |f(x)|\leq M\psi_V(x)\ \text{a.e.\ on}\ \R\right\}.
\end{align*}
We first prove that $\mathcal{C}_{M}^*$ is invariant under the resolvents $\{(1+\ep S)^{-1}\}_{\ep>0}$
 via (a modified) Stampacchia's truncation method. 

Let $f\in \mathcal{C}_M^*$ be arbitrary 
and put $u=(1+\ep S)^{-1}f\in D(S)=H^2(\R)\subset C^1(\R)$. 
Then we fix a nonnegative function 
$\eta\in C^1(\R)$ satisfying $0\leq \eta'\in L^\infty(\R)$ 
and 
\[
\begin{cases}
\eta(s)=0 & \text{if}\ s\leq 0, 
\\
\eta'(s)>0 & \text{if}\ s> 0, 
\end{cases}
\]
and $v=u-M\psi_V\in C^1(\R)$.
Note that 
$u\in L^\infty(\R)$ gives $v(x)\to -\infty$ as $|x|\to \infty$.
Multiplying $\eta(v)\in C^1_0(\R)$ to 
the equality $f-u=\ep\mathcal{S}v$, we have  
\begin{align*}
\int_{\R}(f-u)\eta(v)\,dx
=
\ep\int_{\R}(v')^2\eta'(v)
+Vv\eta(v)\,dx
\geq0.
\end{align*}
Therefore
\[
\int_\R v\eta(v)\,dx
=
\int_\R (u-M\psi_V)\eta(v)\,dx
\leq 
\int_\R (f-M\psi_V)\eta(v)\,dx
\leq 0
\]
which means that $u\leq M\psi_V$ a.e.\ on $\R$.
The other inequality can be verified by 
the same argument with $u$ replaced with $-u$.
This implies $(1+\ep S)^{-1}\mathcal{C}_M^*\subset \mathcal{C}_M^*$. 

Next we show the following representation  
\[
\mathcal{C}_{m}=
\left\{f\in L^2(\R)\;;\;\int_{\R}fg\,dx\leq 1\ (\forall g\in \mathcal{C}_{1/m}^*)\right\}.
\]
If $f\in L^2(\R)$ satisfies $\int_{\R}fg\,dx\leq 1$ for all $g\in\mathcal{C}_{1/m}^*$,
then 
choosing $g_n=\frac{n \psi_V}{m(n|f|+\psi_V)}f\in \mathcal{C}_{1/m}^*$,
we can find that 
\begin{align*}
\int_\R\frac{|f|^2}{|f|+n^{-1}\psi_V}\psi_V\,dx=m\int_\R fg_n\,dx\leq m.
\end{align*}
Letting $n\to \infty$, by the monotone convergence theorem we have $f\in \mathcal{C}_m$. 
Conversely, if $f\in \mathcal{C}_m$,  then 
for every $g\in \mathcal{C}_{1/m}^*$, 
\begin{align*}
\int_{\R}fg\,dx= \int_\R (\psi_V f)(\psi_V^{-1}g)\,dx
\leq 
\frac{1}{m}
\|\psi_V f\|_{L^1}\leq 1.
\end{align*}

To close the proof, we now prove $(1+\ep S)^{-1}\mathcal{C}_m\subset \mathcal{C}_m$ 
which enables us to conclude $e^{-tS}\mathcal{C}_m\subset \mathcal{C}_m$ 
via so-called Post--Widder inversion formula
\[
\lim_{n\to \infty}\left(1+\frac{t}{n}S\right)^{-n}f=e^{-tS}f
\]
(see e.g., Engel--Nagel \cite{ENbook}) with the closedness of $\mathcal{C}_m$.
If $f\in \mathcal{C}_m$, then by 
the selfadjointness of $(1+\ep S)^{-1}$ and
the invariance property $(1+\ep S)^{-1}\mathcal{C}_{1/m}^*\subset \mathcal{C}_{1/m}^*$,
we have
for every $g\in \mathcal{C}_{1/m}^*$, 
\[
\int_\R \big((1+\ep S)^{-1}f\big)g\,dx
=
\int_\R f\big((1+\ep S)^{-1}g\big)\,dx\leq 1
\]
which gives $(1+\ep S)^{-1}f\in \mathcal{C}_m$. The proof is complete.
\end{proof}

Here we state a weighted version of similarities of the Nash inequality
for the Schr\"odinger operator $\mathcal{S}$.
\begin{lemma}\label{lem:nash-pre}
There exists a positive constant $C_{{\rm Nash}}^*$ such that 
for every $f\in D(S^{1/2})\cap L^1(\R,d\mu_V)$, 
\[
\|\psi_V^{2/3} f\|_{L^2}^6
\leq 
C_{{\rm Nash}}^*\|\psi_Vf\|_{L^1}^{4}
\|S^{1/2}f\|_{L^2}^{2}.
\]
\end{lemma}
\begin{proof}
Let $f\in D(S^{1/2})\cap L^1(\R,d\mu_V)$ be fixed.
Fix $\zeta\in C_0^1(\R)$ satisfying $\zeta\geq 0$ and  
${\rm supp}\,\zeta\subset[\frac{1}{2},2]$. 
For $R>0$, we put 
\[
f_R(x)=f(x)\zeta\left(\frac{\psi_V(x)}{R}\right), \quad x\in \R. 
\]
Then noting that $\frac{R}{2}\leq \psi_V(x)\leq 2R$ on $\supp f_R$, 
we have $R\|f_R\|_{L^1}\leq 2\|\psi_Vf\|_{L^1}$ and also
for $p\geq 1$ and $q\in\R$,  
\begin{align*}
\int_0^\infty\|f_R\|_{L^p}^p\,R^{q-1}dR
&=
\int_{\R}|f(x)|^p\int_0^\infty \left|\zeta\left(\frac{\psi_V(x)}{R}\right)\right|^pR^{q-1}\,dR\,dx
\\
&=
K_{p,q}\int_{\R}|f(x)|^p\psi_V(x)^q\,dx,
\end{align*}
where we have put 
\[
K_{p,q}=
\int_0^\infty |\zeta(s)|^p s^{-q-1}\,ds.
\]
On the other hand, the relation $\|\psi_1'\|_{L^\infty}=\|\psi_2'\|_{L^\infty}=k_V$ 
gives 
\begin{align*}
\left|
f'(x)
\zeta\left(\frac{\psi_V(x)}{R}\right)
- f_R'(x)
\right|
&=
|f(x)|\left|\zeta'\left(\frac{\psi_V(x)}{R}\right)\right|\frac{|\psi_V'(x)|}{R}
\\
&\leq 
4k_V\|\zeta'\|_{L^\infty(\R)}\frac{|f(x)|}{\psi_V(x)},
\end{align*}
and therefore, using Lemma \ref{lem:hardy:V} {\bf (i)} we have
\begin{align*}
\|f_R'\|_{L^2}
&\leq \|f'\zeta(\psi_V/R)\|_{L^2}
+\|f'\zeta(\psi_V/R)-f_R'\|_{L^2}
\\
&\leq \|\zeta\|_{L^\infty} 
\|f'\|_{L^2}+
4k_V\|\zeta'\|_{L^\infty}\left\|\frac{1}{\psi_1}+\frac{1}{\psi_2}\right\|_{L^\infty}
\left\|\frac{f}{\psi_1\psi_2}\right\|_{L^2}
\\
&\leq 
\widetilde{K}
\|S^{1/2}f\|_{L^2}, 
\end{align*}
where $\widetilde{K}=\|\zeta\|_{L^\infty}+8\|\zeta'\|_{L^\infty}\|\psi_1^{-1}+\psi_2^{-1}\|_{L^\infty}$.
Therefore 
combining all inequalities in the above consideration
with the one-dimensional Nash inequality (Lemma \ref{lem:nash}), 
we conclude that 
\begin{align*}
K_{2,4/3}
\|f\psi_V^{2/3}\|_{L^2}^2
&=
\int_0^\infty\|f_R\|_{L^2}^2R^{1/3}\,dR
\\
&\leq 
C_{\rm Nash}^{1/3}\int_0^\infty
\|f_R\|_{L^1}^{4/3}\|f_R'\|_{L^2}^{2/3}R^{1/3}dR
\\
&\leq 
(2C_{\rm Nash})^{1/3}
\widetilde{K}^{2/3}
\left(\int_0^\infty
\|f_R\|_{L^1}\,dR\right)
\|\psi_Vf\|_{L^1}^{1/3}
\|S^{1/2}f\|_{L^2}^{2/3}
\\
&\leq 
(2C_{\rm Nash})^{1/3}
\widetilde{K}^{2/3}K_{1,1}
\|\psi_Vf\|_{L^1}^{4/3}
\|S^{1/2}f\|_{L^2}^{2/3}.
\end{align*}
The proof is complete.
\end{proof}

The following Proposition is 
a variant of Nash inequality for the Schr\"odinger operator $\mathcal{S}$ 
which is explained in Introduction. 

\begin{proposition}\label{prop:nash-type}
There exists a positive constant $C_{{\rm Nash}}^\sharp$ 
such that 
for every $f\in D(S^{1/2})\cap L^1(\R,d\mu_V)$, 
\[
\|f\|_{L^2}^{2+4/3}
\leq 
C_{{\rm Nash}}^\sharp
\|\psi_Vf\|_{L^1}^{4/3}
\|S^{1/2}f\|_{L^2}^{2}.
\]
\end{proposition}
\begin{proof}
We combine the inequalities in 
Lemma \ref{lem:hardy:V} {\bf (i)} 
and \ref{lem:nash-pre} with the relation 
$\psi_1\psi_2\leq C\psi_V$
and 
the inequality
\[
\|f\|_{L^2}\leq 
\left\|\frac{f}{\psi_V}\right\|_{L^2}^{2/5}
\|\psi_V^{2/3}f\|_{L^2}^{3/5}
\leq 
C^{-2/5}
\left\|\frac{f}{\psi_1\psi_2}\right\|_{L^2}^{2/5}
\|\psi_V^{2/3}f\|_{L^2}^{3/5}.
\]
The proof is complete.
\end{proof}

\section{Decay estimate for the Schr\"odinger semigroup $e^{-tS}$}\label{sec:semigroup}

By  applying Proposition \ref{prop:nash-type}, 
we discuss a decay property of solutions to 
the parabolic equation governed by the Schr\"odinger operator $S$:
\begin{equation}\label{eq:heat}
\begin{cases}
\pa_t v(x,t)+Sv(x,t)=0
&\text{in}\ \R\times (0,\infty),
\\
v(x,0)=f(x) 
&\text{on}\ \R;
\end{cases}
\end{equation}
note that the solution $v$ of \eqref{eq:heat} 
can be represented by $v(t)=e^{-tS}f$. 
\subsection{$L^2$-decay estimate for $e^{-tS}$}
The first subsection provides $L^2$-$L^1$ type decay estimate for $e^{-tS}$
as a role of ``Nash inequality'' (Proposition \ref{prop:nash-type}).
The proof is analogous to the $L^2$-$L^1$ estimate for the heat semigroup via 
the usual Nash inequality. 
\begin{proposition}\label{lem:L2-wL1}
Assume that $V\in\mathcal{V}$. Then 
there exists a positive constant $C_2$ such that
for every $f\in L^2(\R)\cap L^1(\R,d\mu_V)$,   
the corresponding solution $v$ of the parabolic equation \eqref{eq:heat} 
satisfies for every $t>0$,  
\[
\|v(t)\|_{L^2}
\leq C_2t^{-3/4}\|\psi_Vf\|_{L^1}.
\]
\end{proposition}
\begin{proof}
The case $f\equiv 0$ is trivial. 
We assume $f\not\equiv 0$. 
Put $Z(t)=\|v(t)\|_{L^2}^2$ for $t\geq 0$. 
If $Z(t_0)=0$ for some $t_0>0$, then by the analyticity of 
$e^{-tS}$, we deduce $e^{-tS}f=0$ $(t>0)$ 
and then the continuity of $t\mapsto e^{-tS}f$ implies $f\equiv 0$ 
which is contradiction. Therefore we can always assume $Z(t)>0$. 
Moreover, Lemma \ref{lem:C_M-inv} provides
\begin{align*}
0<\|\psi_Vv(t)\|_{L^1}\leq \|\psi_Vf\|_{L^1}.
\end{align*}
Hence by Proposition \ref{prop:nash-type}, we have
\begin{align*}
\frac{d}{dt}Z(t)^{-2/3}
&= -\frac{2}{3}Z(t)^{-5/3}Z'(t)
\\
&= \frac{4}{3}Z(t)^{-5/3}
\|\psi_V v(t)\|_{L^1}^{-4/3}\Big(\|\psi_V v(t)\|_{L^1}^{4/3}
\|S^{1/2}v(t)\|_{L^2}^2\Big)
\\
&\geq \frac{4}{3(C_{{\rm N},0}^\sharp)^{10/3}\|\psi_Vf\|_{L^1}^{4/3}}.
\end{align*} 
Integrating it over $[0,t]$, we deduce
\[
Z(t)^{-2/3}\geq 
Z(0)^{-2/3}
+\frac{4t}{3(C_{{\rm N},0}^\sharp)^{10/3}\|\psi_Vf\|_{L^1}^{4/3}},
\]
and therefore we obtain the desired inequality.
\end{proof}
\subsection{Upper bounds for the Dirichlet energy with weights}

We also consider the decay estimate for the Dirichlet energy with weights. 
The Hardy--Rellich type inequality for the Schr\"odinger operator $S$ is employed.

\begin{proposition}\label{lem:w-energy-decay}
Moreover, for every 
$\beta\in [0,1)$,
there exists a positive 
constant $C_{3,\beta}$ such that 
for every $t>0$, 
\[
\int_{\R}\frac{(\pa_xv)^2+Vv^2}{\psi_V^{\beta}}\,dx
\leq C_{3,\beta} t^{-\frac{5+\beta}{2}}\|\psi_Vf\|_{L^1}^2.
\]
\end{proposition}

\begin{proof}
We use Lemma \ref{lem:hardy:V} {\bf (iii)}
with 
$\tau \|Se^{-\tau S}\|_{L^2\to L^2}\leq 1$ ($\tau>0$).
Then we deduce
\begin{align*}
\int_{\R}\frac{(\pa_xv)^2+Vv^2}{\psi_V^{\beta}}\,dx
&\leq 
C_{1,3}\|Sv(t)\|_{L^2}^{1+\frac{\beta}{2}}\|v(t)\|_{L^2}^{1-\frac{\beta}{2}}
\\
&\leq 
C_{1,3}\Big(\frac{2}{t}\Big)^{1+\frac{\beta}{2}}
\|v(t/2)\|_{L^2}^2.
\end{align*}
Applying Proposition \ref{lem:L2-wL1}, we obtain the desired estimate.
\end{proof}

\subsection{Optimality for the decay estimates}

Next we provide the proof of 
decay estimates from below 
in a weak sense.  
The following lemma provides a lower bound for the decay rate of 
$L^2$-norm and the Dirichlet energy. 
The proof depends on the test function method 
with positive harmonic functions with respect to $\mathcal{S}$.
\begin{proposition}\label{prop:lower}
Let  $f\in L^2(\R)\cap L^1(\R,d\mu_V)$ and let $v$ be the solution of 
parabolic equation \ref{eq:heat}.  
If 
\[
\int_{\R}fd\mu_V\neq 0,
\]
then there exists positive constants $c_4, c_4'$ such that for every $t\geq 0$, 
\begin{align*}
\int_0^t\|v(\tau)\|_{L^2}^{4/3}\,d\tau 
\geq c_4\log (1+t), \quad 
\int_0^t\|S^{1/2}v(\tau)\|_{L^2}^{4/5}\,d\tau 
\geq c_4'\log (1+t).
\end{align*}
In particular, one has
\begin{gather*}
\int_0^t(1+\tau)^{3/2}\|S^{1/2}v(\tau)\|_{L^2}^2\,d\tau
\geq (c_4')^{5/2}\log (1+t), 
\\
\int_0^t(1+\tau)^{1/2}\|v(\tau)\|_{L^2}^{2}\,d\tau
\geq (c_4)^{3/2}\log (1+t).
\end{gather*}
\end{proposition}

\begin{proof}
It suffices to consider the case where 
\[
\int_{\R}fd\mu_V> 0.
\]
We fix 
$\eta\in C^\infty(\R)$ 
as 
\[
\eta(s)
\begin{cases}
=1
& \text{if}\ s\leq \frac{1}{2},
\\
\text{is decreasing} 
& \text{if}\ \frac{1}{2}<s<1,
\\
=0
& \text{if}\ s\geq 1,
\end{cases}
\]
and for $R>0$, 
\[
\eta_R(x,t)=\eta\big(\xi_R(x,t)\big), \quad
\xi_R(x,t)=\frac{|x|^2+t}{R}, 
\quad (x,t)\in \R \times [0,\infty).
\]
We also use  the indicator function $\chi$ of the interval $(\frac{1}{2},1)$ which satisfies 
$|\eta'|\leq \|\eta'\|_{L^\infty}\chi$.
By the dominated convergence theorem, we can choose $R_0>0$ such that 
for every $R\geq R_0$, 
\[
\int_\R f\eta\left(\frac{|x|^2}{R}\right)\psi_V\,dx
\geq \frac{1}{2}
\int_\R f\,d\mu_V>0.
\]
Then multiplying 
$\psi_V\eta_R$ to the equation in \eqref{eq:heat} 
and integrating it over $\R\times (0,\infty)$, we have
\begin{align*}
0&=\int_0^\infty
\!\!\int_\R (\pa_tv+S_Vv)\eta_R\psi_V\,dx\,dt
\\
&=\int_0^\infty
\frac{d}{dt}\left(\int_\R v\eta_R\psi_V\,dx\right)\,dt
+
\int_0^\infty
\!\!\int_\R v \big(-\pa_t(\eta_R\psi_V)+
\rev{S} 
(\eta_R\psi_V)\big)\,dx\,dt
\\
&=
-
\int_\R f\eta\left(\frac{|x|^2}{R}\right)\psi_V\,dx
+
\int_0^\infty
\!\!\int_\R v \left(-\pa_t\eta_R
+2\frac{\psi_V'\pa_x\eta_R}{\psi_V}+\pa_x^2\eta_R\right)\psi_V\,dx\,dt
\\
&\leq 
-
\int_\R f\eta\left(\frac{|x|^2}{R}\right)\psi_V\,dx
+
\frac{\widetilde{K}_1}{R}
\int_0^\infty
\!\!\int_\R |v| 
\chi(\xi_R)\psi_V\,dx\,dt, 
\end{align*}
where $\widetilde{K}_1=(2+4\|x\psi_V'/\psi_V\|_{L^\infty})\|\eta'\|_{L^\infty}+4\|\eta''\|_{L^\infty}$. 
Since the H\"older inequality yields 
\begin{align*}
\int_0^\infty
\!\!\int_\R |v| 
\chi(\xi_R)\psi_V\,dx\,dt
&
\leq 
\left(\int_0^\infty
\!\!\int_\R\chi(\xi_R)\psi_V^2\,dx\,dt\right)^{2/5}
\left(\int_0^\infty
\!\!\int_\R|v|^{5/3}\chi(\xi_R)\psi_V^{1/3}\,dx\,dt\right)^{3/5}, 
\\
&
\leq 
\widetilde{K}_2R
\left(\int_0^\infty
\!\!\int_\R|v|^{5/3}\chi(\xi_R)\psi_V^{1/3}\,dx\,dt\right)^{3/5}
\end{align*}
for some positive constant $\widetilde{K}_2$. 
Then we deduce for $R\geq R_0$, 
\[
\left(\int_{\R}f\,d\mu_V\right)^{5/3}
\leq 
\big(2\widetilde{K}_1\widetilde{K}_2\big)^{5/3}
\int_0^\infty
\!\!\int_\R|v|^{5/3}\chi(\xi_R)\psi_V^{1/3}\,dx\,dt.
\]
We put an auxiliary function $Y:[R_0,\infty)\to [0,\infty)$ as 
\[
Y(R)=\int_{0}^R
\left(\int_0^\infty
\!\!\int_\R|v|^{5/3}\chi(\xi_R)\psi_V^{1/3}\,dx\,dt\right)\,\frac{d\rho}{\rho}, \quad R\geq R_0.
\]
Then the previous computation shows 
the ordinary differential inequality
\[
\big(2\widetilde{K}_1\widetilde{K}_2\big)^{-5/3}
\left(\int_\R f\,d\mu_V\right)^{5/3}
\frac{1}{R}\leq Y'(R), \quad R\geq R_0
\]
which implies
\[
Y(R_0)+
\big(2\widetilde{K}_1\widetilde{K}_2\big)^{-5/3}
\left(\int_\R f\,d\mu_V\right)^{5/3}
\log \left(\frac{R}{R_0}\right)\leq Y(R), \quad R\geq R_0.
\]
On the other hand, the Fubini theorem gives 
\begin{align*}
Y(R)
&=
\int_0^\infty
\!\!\int_\R |v|^{5/3}
\left(\int_{0}^R\chi(\xi_\rho)\,\frac{d\rho}{\rho}\right)\psi_V^{1/3} \,dx\,dt
\\
&=
\int_0^\infty
\!\!\int_\R |v|^{5/3}
\left(\int_{\xi_R}^\infty 
\chi(\sigma)\frac{d\sigma}{\sigma}\right)
\psi_V^{1/3}\,dx\,dt
\\
&\leq 
\log 2\int_0^R\|\psi_V^{1/5}v(t)\|_{L^{5/3}}^{5/3} \,dt.
\end{align*}
Combining the above inequalities with the interpolation
\[
\|\psi_V^{1/5}v(t)\|_{L^{5/3}}
\leq 
\|\psi_V v(t)\|_{L^1}^{1/5}
\|v(t)\|_{L^2}^{4/5}
\]
and Lemma \ref{lem:C_M-inv}, we obtain the first desired estimate. 
The second estimate is easily obtained via the first estimate 
with Proposition \ref{prop:nash-type} and Lemma \ref{lem:C_M-inv} again. 
The proof is complete.
\end{proof}

\section{Optimal decay estimate for the target problem}\label{sec:damped}

To end the paper, we prove 
the decay estimates 
for \eqref{problem} from above and below. 
To clarify the situation, 
it is reasonable to use the notation
$U(t)g$ which stands for 
the unique solution of the following 
Cauchy problem:
\begin{equation}\label{eq:dw:0,g}
\begin{cases}
\pa_t^2w+
\rev{S} 
w+\pa_tw=0
&\text{in}\ \R\times (0,\infty), 
\\
(w,\pa_tw)(0)=(0,g);
\end{cases}
\end{equation}
note that it is now clear that the abstract diffusion phenomena (Lemma \ref{lem:matsumura-abst}) is applicable. 
Then we also introduce a function
\begin{align}
\label{eq:diff-ini}
v_0=u_0+u_1.
\end{align}

The following lemma gives a decay rate for two important quantities of the difference between 
$u(t)$ and $U(t)v_0$. 
\begin{lemma}\label{lem:5.1}
Assume that $V\in \mathcal{V}$ and $(u_0,u_1)\in \mathcal{H}$. Let $u$ be the solution of \eqref{problem}
and let $v_0$ be given in \eqref{eq:diff-ini}.  Then 
there exists a positive constant $C_{5}$ such that 
for every $t\geq 0$, 
\begin{align*}
(1+t)^3E\big(u-U(\cdot)v_0;t\big)
+
(1+t)^2\|u(t)-U(t)v_0\|_{L^2}
&\leq C_{5}\Big(\|S^{1/2}u_0\|_{L^2}^2+\|u_0\|_{L^2}^2\Big).
\end{align*}
\end{lemma}
\begin{proof}
Put $\widetilde{u}(t)=U(t)u_0$. Then it is easy to see that 
$\pa_t\widetilde{u}(t)=u(t)-U(t)v_0$.
Applying the usual energy method (for higher order derivatives), 
we have
\begin{align*}
(1+t)^3E(\pa_t\widetilde{u};t)
+
(1+t)^2\|\pa_t\widetilde{u}(t)\|_{L^2}^2
\leq C
\Big(
E(\pa_t\widetilde{u};0)+\|\pa_t\widetilde{u}(0)\|_{L^2}^2
+
E(\widetilde{u};0)+\|\widetilde{u}(0)\|_{L^2}^2
\Big).
\end{align*}
The proof is complete.
\end{proof}

Here we focus our attention to the other part $U(t)v_0$. 
Roughly speaking, the optimal decay rate for $U(t)v_0$ is given by that of $e^{-tS}v_0$. 
\begin{lemma}\label{lem:5.2}
Assume that $V\in \mathcal{V}$ and $(u_0,u_1)\in \mathcal{H}$. 
Let $v_0$ be given in \eqref{eq:diff-ini}. 
Then there exists a positive constant $C_6$ such that 
for every $t\geq 1$, 
\begin{equation}
t^3 E\big(U(\cdot)v_0-e^{-tS}v_0;t\big)
+
t^2 \|U(t)v_0-e^{-tS}v_0\|_{L^2}^2
\leq 
C_{6}\|v_0\|_{L^2}^2.
\end{equation}
\end{lemma}
\begin{proof}
Applying Lemma \ref{lem:matsumura-abst} with $A=S$, we have 
for every $t\geq1$, 
\begin{align*}
\|U(t)v_0-e^{-tS_V}v_0\|_{L^2}
&\leq 
C_{{\rm M},1}\Big(t^{-1}+e^{-t/16}\Big)\|v_0\|_{L^2},
\\
\|S^{1/2}(U(t)v_0-e^{-tS_V}v_0)\|_{L^2}
&\leq 
C_{{\rm M},2}\Big(t^{-\frac{3}{2}}+e^{-t/16}\Big)\|v_0\|_{L^2},
\\
\|\pa_t(U(t)v_0-e^{-tS_V}v_0)\|_{L^2}
&\leq 
2C_{{\rm M},3}t^{-2}\|v_0\|_{L^2}.
\end{align*}
The proof is complete.
\end{proof}

\begin{proof}[Proof of Theorem \ref{main1}]
Assume $\lr{x}(u_0+u_1)\in L^1(\R)$ which is equivalent to $v_0\in L^1(\R,d\mu_V)$. 
Put for $t\geq0$, 
\[
\pa_t\widetilde{u}(t)=u(t)-U(t)v_0, \quad \widetilde{v}(t)=U(t)v_0-e^{-tS}v_0, 
\quad 
v(t)=e^{-tS}v_0. 
\]
Then 
applying Lemmas \ref{lem:5.1} and \ref{lem:5.2}, we have for every $t\geq 1$, 
\begin{align*}
t \|u(t)-e^{-tS}v_0\|_{L^2}
&\leq 
t\|\pa_t\widetilde{u}(t)\|_{L^2}+t\|\widetilde{v}(t)\|_{L^2}
\\
&\leq 
C_{5}^{1/2}\Big(\|S^{1/2}u_0\|_{L^2}^2+\|u_0\|_{L^2}^2\Big)^{1/2}+C_6^{1/2}\|v_0\|_{L^2}.
\end{align*}
On the other hand, one can check that 
for every $\beta\in [0,1)$ and $t\geq 1$, 
\begin{align*}
\left(
\int_{\R}\frac{\mathcal{E}(u;x,t)}{(\psi_1\psi_2)^{\beta}}\,dx
\right)^{1/2}
&\leq 
\left(
\int_{\R}\frac{\mathcal{E}(\pa_t\widetilde{u};x,t)}{(\psi_1\psi_2)^{\beta}}\,dx
\right)^{1/2}
+
\left(
\int_{\R}\frac{\mathcal{E}(\widetilde{v};x,t)}{(\psi_1\psi_2)^{\beta}}\,dx
\right)^{1/2}
\\
&+
\left(
\int_{\R}\frac{\mathcal{E}(v;x,t)}{(\psi_1\psi_2)^{\beta}}\,dx
\right)^{1/2}
\\
&\leq 
\big( E(\pa_t\widetilde{u};t) \big)^{1/2}
+
\big(E(\widetilde{v};t)\big)^{1/2}
+
\left(
\int_{\R}\frac{\mathcal{E}(v;x,t)}{(\psi_1\psi_2)^{\beta}}\,dx
\right)^{1/2}.
\end{align*}
By Lemma \ref{lem:hardy:V} {\bf (iii)},
the last term on the right-hand side of the above inequality can be estimated 
as 
\begin{align*}
\int_{\R}\frac{\mathcal{E}(v;x,t)}{(\psi_1\psi_2)^{\beta}}\,dx
&\leq 
\|\pa_tv\|_{L^2}^2
+
\int_{\R}\frac{(\pa_xv)^2+Vv^2}{(\psi_1\psi_2)^{\beta}}\,dx
\\
&\leq 
\|Sv\|_{L^2}^2+C_{1,3}\|Sv\|_{L^2}^{1+\beta/2}\|v\|_{L^2}^{1-\beta/2}
\\
&\leq 
\|S v\|_{L^2}^{2}
+
C_{1,3}\|Sv\|_{L^2}^{1+\beta/2}\|v\|_{L^2}^{1-\beta/2}.
\end{align*}
The analyticity of $e^{-tS}$ and Lemma \ref{lem:L2-wL1} give 
for every $t\geq 1$, 
\[
\int_{\R}\frac{\mathcal{E}(v;x,t)}{(\psi_1\psi_2)^{\beta}}\,dx
\leq 
\widetilde{K}_3 
\Big(
t^{-7/2}
+
t^{-(5+\beta)/2}\Big)\|\psi_Vv_0\|_{L^1}^2.
\]
Combining the above inequalities with 
Lemmas \ref{lem:5.1} and \ref{lem:5.2}, 
we consequently reach  the desired inequalities. 
\end{proof}
\begin{proof}[Proof of Theorem \ref{main2}]
We finally prove the optimality of the decay rates of
$\|u(t)\|_{L^2}$ and $E(u;t)$ via the same decomposition in the proof of Theorem \ref{main1}. 
\begin{align*}
&\left(
\int_1^t(1+\tau)^{1/2}\|u(\tau)\|_{L^2}^{2}\,d\tau
\right)^{1/2}
\\
&\geq 
\left(
\int_1^t(1+\tau)^{1/2}\|v(\tau)\|_{L^2}^{2}\,d\tau
\right)^{1/2}
-
\left(
\int_1^t(1+\tau)^{1/2}\|u(\tau)-v(\tau)\|_{L^2}^{2}\,d\tau
\right)^{1/2}
\\
&\geq
\left(
\int_1^t(1+\tau)^{1/2}\|v(\tau)\|_{L^2}^{2}\,d\tau
\right)^{1/2}
 -
\left(
\widetilde{K}_3\int_1^\infty(1+\tau)^{1/2}\tau^{-2}\,d\tau
\right)^{1/2}.
\end{align*}
Therefore Proposition \ref{prop:lower} implies the logarithmic growth in $t$ of 
$\int_0^t(1+\tau)^{1/2}\|u(\tau)\|_{L^2}^{2}\,d\tau$.
For the case of the energy functional is similar:
\begin{align*}
&\left(
\int_1^t(1+\tau)^{3/2}E(u;\tau)\,d\tau
\right)^{1/2}
\\
&\geq 
\left(
\int_1^t(1+\tau)^{3/2}E(v;\tau)\,d\tau
\right)^{1/2}
-
\left(
\int_1^t\tau^{3/2}E(u-v;\tau)\,d\tau
\right)^{1/2}
\\
&\geq 
\left(
\int_1^t(1+\tau)^{3/2}E(v;\tau)\,d\tau
\right)^{1/2}
-
\left(
\widetilde{K}_5\int_1^\infty(1+\tau)^{3/2}\tau^{-3}\,d\tau
\right)^{1/2}.
\end{align*}
The proof is complete.
\end{proof}

\subsection*{Declarations}
\begin{description}
\item[Data availability] Data sharing not applicable to this article as no datasets were generated or analysed during the current study.
\item[Conflict of interest] The author declares that he has no conflict of interest.
\end{description}

\end{document}